\documentclass{article}
\usepackage{arxiv}
\usepackage{lineno,hyperref}
\modulolinenumbers[1]
\usepackage{amsmath}
\usepackage{graphicx}
\usepackage{amsfonts}
\usepackage{float}
\usepackage{verbatim}
\usepackage{amsthm}
\usepackage{amssymb}
\usepackage{latexsym}
\usepackage{epstopdf}
\usepackage{color}
\usepackage{a4}
\usepackage{amsmath}
\allowdisplaybreaks[4]

\makeatletter
\@addtoreset{equation}{section}

\theoremstyle{definition}%
\newtheorem{theorem}{Theorem}[section]
\newtheorem{lemma}{Lemma}[section]
\newtheorem{assumption}{Assumption}[section]
\theoremstyle{definition}
\newtheorem{example}{Example}[section]
\newtheorem{remark}{Remark}%

\title{Numerical approximations to invariant measures of hybrid stochastic differential equations with superlinear coefficients via the backward Euler-Maruyama method}

\date{}

\author{
Wei Liu\\
Department of Mathematics\\
Shanghai Normal University\\
100 Guilin Road, Shanghai, 200234, China\\
\texttt{weiliu@shnu.edu.cn} 
\And
Jie Xu\\
Department of Mathematics\\
Shanghai Normal University\\
100 Guilin Road, Shanghai, 200234, China\\
\texttt{jiexu\_math@hotmail.com} 
}

\renewcommand{\shorttitle}{BEM method for invariant measure of super-linear hybrid SDEs}

\begin{document}
\maketitle

\begin{abstract}
For stochastic differential equations (SDEs) with Markovian switching, whose drift and diffusion coefficients are allowed to contain superlinear terms, the backward Euler-Maruyama (BEM) method is proposed to approximate the invariant measure. The existence and uniqueness of the invariant measure of the numerical solution generated by the BEM method is proved. Then the convergence of the numerical invariant measure to its underlying counterpart is shown. Those results obtained in this work release the requirement of the global Lipschitz condition on the diffusion coefficient in [X. Li et al. SIAM J. Numer. Anal. 56(3)(2018), pp. 1435-1455] and can also be regarded as a non-trivial extension of [W. Liu et al. Appl. Numer. Math. 184(2023), pp. 137-150] to the case of hybrid SDEs.
\end{abstract}


\section{Introduction}
Stochastic differential equations (SDEs) with Markovian switching, sometimes also called hybrid SDEs, are a set of SDEs may switch from one to another controlled by a Markov process as the time goes on. Detailed introductions and investigations of this type of stochastic systems can be found in monographs \cite{MaoYuan2006Book,YinZhu2010Book}.
\par
Such a hybrid stochastic system has broad applications in ecology \cite{BaoShao2016}, epidemiology \cite{GrayGreenhalghMaoPan2012}, finance \cite{Liu2010}, immunology \cite{DengLiu2020}, and many other areas where uncertain phenomena need to be modeled.
\par
To apply hybrid SDEs in practice, numerical simulations of them are necessary, as it is hard to obtain explicit forms of solutions to most hybrid SDEs. Even for some rare cases that explicit expressions of the true solutions are achievable, simulations still have to be done to sample from the solutions, as various follow-up estimations and tests may be conducted. 
\par
When we try to simulate the hybrid SDEs, the choices of numerical methods are essential, as improper numerical methods could lead to that numerical solutions do not preserve properties of true solutions \cite{HighamMaoYuan2007,HutzenthalerJentzenKloeden2011}.
\par
Therefore, careful numerical analyses of numerical methods for SDEs with Markovian switching have been attracting lots of attention in recent years. The finite time convergence and the asymptotic behaviors are two main topics for numerical methods. Since our focus in this work is numerical invariant measures, in the next three paragraphs we give more detailed literature reviews of the asymptotic behaviors of numerical methods as they are more relevant. For the studies of the finite time convergence, we refer the readers to \cite{KumarKumar2020,MaoYuanYin2007,NguyenNguyenHoangYin2018,NguyenHoangNguyenYin2017,ZhangSongLiu2023} and reference therein.
\par
Among all kinds of asymptotic behaviors of numerical methods, numerical invariant measure is one of most interesting properties that researchers have been investigating. Briefly speaking, under some conditions the underlying equation has a unique invariant measure, which however rarely has an explicit expression. Hence, finding some proper numerical solution that also has a unique invariant measure and proving the numerical one to be convergent to the underlying counterpart are essential. 
\par
when both the drift and diffusion coefficients of SDEs with Markovian switching satisfy the global Lipschitz condition, the existence and uniqueness of the invariant measure for the numerical solution generated by the Euler-Maruyama (EM) method is proved in \cite{MaoYuanYin2005} and the convergence of the numerical invariant measure to the underlying one is shown in \cite{YuanMao2005}. In addition, the numerical invariant measure of the EM method for SDEs with different types of Markovian switching are discussed in \cite{BaoShaoYuan2016}. Meanwhile, the numerical invariant measure of the EM method for the case of the state-dependent switching is studied in \cite{Shao2018}. 
\par
When some superlinear terms are allowed in coefficients, the EM may not work. Therefore, the alternative numerical methods are required. The modified explicit Euler type methods and the implicit methods are two main candidates. The truncated EM method, as a modified Euler type method, is proposed for approximating the invariant measures of hybrid SDEs when both the drift and diffusion coefficients are allowed to contain some superlinear terms \cite{YangLi2018}. Simultaneously, the backward Euler-Maruyama (BEM) method, as an implicit method, is investigated to approximate the invariant measure of hybrid SDEs, whose drift coefficient may contain superlinear terms but diffusion coefficient is still required to obey the linear growth condition \cite{LiMaYangYuan2018}. However, when the diffusion coefficient also contains some superlinear terms, it is still not clear whether the BEM method is capable to approximate the invariant measure of hybrid SDEs.
\par
In this work, we give the positive answer with rigorous numerical analysis. For hybrid SDEs with the superlinearly growing drift and diffusion coefficients, we prove the existence and uniqueness of the invariant measure of the numerical solution generated by the BEM method firstly in Section \ref{subsec:nim}. Then, we show the convergence of the numerical invariant measure to the underlying counterpart in Section \ref{subsec:con}. Numerical simulations are provided to demonstrate our theoretical results in Section \ref{sec:ns}.
\par
It is not too late to mention at the end of this section that numerical approximations to invariant measures of SDEs without the switching facility embedded have also been attracting a considerable amount of attention in recent years \cite{ChenGan2023,JiangWengLiu2020,LiuMao2015,LiuMaoWu2023,LiuLiu2025,PangWangWu2024}. Compared with the case of SDEs without Markovian switching, studies on SDEs with  Markovian switching need to take the effects of the Markov process controlling the switching facility into consideration, which is nontrivial and brings more difficulties in the analysis. One of interesting phenomena observed from the case of SDEs with  Markovian switching is that the whole system could have a unique invariant measure even if some of the subsystems may not have invariant measures.

\section{Mathematical preliminaries}\label{sec:MathPre}
This section contains basic settings, notations and useful lemmas for the rest of the paper.
\par
Throughout this paper, we work with the complete probability space $(\Omega, \mathcal{F}, \mathbb{P})$. Let $|a|$ denote the Euclidean norm if $a$ is a vector and the Hilbert-Schmidt norm if $a$ is a matrix. The inner product of vectors $a$ and $b$ in the Euclidean space is denoted by $\langle a,b\rangle$. The larger one between scalars $a$ and $b$ is denoted by $a\vee b$ , and the smaller one is denoted by $a\wedge b$. For a vector or matrix $a$, $a^T$ denotes its transpose. For a scalar $a>0$, $[a]$ denotes the largest integer that is smaller than $a$. $\mathbb{N}_+$ denotes the set of non-negative integers. Let $\mathbb{I}_A$ be the indicator function for set $A$.
\par
Let $B(t) = (B_1(t), \cdots, B_d(t))^T$ with $t \geq 0$ be a $d$-dimensional Brownian motion defined on $(\Omega, \mathcal{F}, \mathbb{P})$. Let $r(t)$ with $t \geq 0$ be a right-continuous Markov chain on $(\Omega, \mathcal{F}, \mathbb{P})$ taking values in a finite state space $\mathbb{S} = \{1,2, \cdots, N \}$ with the generator $\Gamma=(\gamma_{ij})_{N \times N}$ given by
\begin{equation*}
\mathbb{P}\{r(t + h) = j \mid r(t) = i\} = 
\begin{cases} 
\gamma_{ij} h + o(h) & \text{if}~ i \neq j, \\
1 + \gamma_{ii} h + o(h) & \text{if}~ i = j,
\end{cases}
\end{equation*}
where $h>0$ is sufficiently small and $\gamma_{ij} \geq 0$ is the transition rate from $i$ to $j$ if $i \neq j$ with $\gamma_{ii}=-\Sigma_{i \neq j} \gamma_{ij}$. In addition, we assume that the transition rate matrix $\Gamma$ is irreducible and conservative. So, the Markov chain $\{r(t)\}_{t \geq 0}$ has a unique stationary distribution $ \mu := (\mu_1, \mu_2, \ldots, \mu_N)$, which can be determined by solving the linear equation
$$
\mu \Gamma = 0 \quad \text{subject to} \quad \sum_{j=1}^N \mu_j = 1~\text{with}~\mu_j \geq 0.
$$

Denote the natural filtration by $ \{\mathcal{F}_t\}_{t \geq 0}$, i.e. $\mathcal{F}_t$ is the $\sigma$-algebra generated by $\{B(s),r(s)\}_{0 \leq s \leq t}$. It is clear that both $B(t)$ and $r(t)$ are $\mathcal{F}_t$-adapted.
\par \noindent
Throughout this paper, we assume that $B(t)$ and $r(t)$ are independent.
\par
We consider the $n$-dimensional SDEs with Markovian switching
\begin{equation}
    \mathrm{d}X(t) = f(X(t), r(t)) \mathrm{d}t + g(X(t), r(t)) \mathrm{d}B(t), \quad t \geq 0,
    \label{theSDE}
\end{equation}
with initial data $X(0) = x_0 \in \mathbb{R}^n$ and $r(0) = i_0 \in \mathbb{S}$, where $f : \mathbb{R}^n \times \mathbb{S} \to \mathbb{R}^n$ and $g : \mathbb{R}^n \times \mathbb{S} \to \mathbb{R}^{n \times d}$. Now we impose several assumptions on f and g as follows.
\begin{assumption}\label{PolyCon}
There exists $ a_i > 0$ and $ q \geq 2$  such that 
\begin{align*}
\left|f(x, i) - f(y, i)\right|^2 \lor \left|g(x, i) - g(y, i)\right|^2 \leq a_i( 1 + \left| x \right| ^{q-2}+ \left| y \right|^{q-2}) \left| x - y\right|^2,
\end{align*}
for all $x, y \in \mathbb{R}^n$, $i \in \mathbb{S}$.
\end{assumption}
From Assumption \ref{PolyCon}, one can straightforwardly derive that for all $x \in \mathbb{R}^n$ and $i \in \mathbb{S}$
\begin{align}
\left| f(x,i) \right| ^2\lor \left| g(x,i) \right|^2\leq b_i \left( 1 + \lvert x \rvert^q \right),
\label{eq:13}
\end{align}
where $b_i =4 a_i + 2\left(\left|f(0, i)\right|^2 \lor \left|g(0, i)\right|^2 \right)$.
\begin{assumption}\label{KhCon}
There exist $ n_i \in \mathbb{R}$  and $l_1 > 4$ such that
\begin{align*}
&2\langle x - y, f(x, i) - f(y, i)\rangle + l_1\lvert g(x, i) - g(y, i)\rvert^2\leq n_i\lvert x - y\rvert^2,
\end{align*}
for all $x, y\in \mathbb{R}^n$, $i\in \mathbb{S}$. 
\end{assumption}

It is not hard to derived from Assumption \ref{KhCon} that 
\begin{align} \label{Khconext}
2\langle x, f(x, i)\rangle + l_2\lvert g(x, i)\rvert^2\leq m + ( 1 + n_i)\lvert x\rvert^2,
\end{align}
where $l_2 = l_1 - 2$ and $m = \max_{i \in \mathbb{S}} \{|f(0,i)|^2 + \frac{l_1l_2}{2}|g(0,i)^2|\}$.

\par
Then, we present our requirement on the switching process.
\begin{assumption}\label{invCon}
For the unique stationary distribution $\mu$, there exist constants $\lambda_1 >0$ and $\lambda_2 >0$ such that
\begin{equation*}
\sum_{{j \in \mathbb{S}}}^{} \mu_j \frac{n_j+1}{1 - (n_j+1)/(n_M+2)}  = - \lambda_1 \quad \text{and} \quad \sum_{{j \in \mathbb{S}}}^{} \mu_j \frac{n_j}{1 - n_j/(n_M+2)}  = - \lambda_2, 
\end{equation*}
where  
$n_M = \max_{i \in \mathbb{S}}{|n_i|}$.
\end{assumption}


To emphasize the initial value, $(X^{x_0,i_0}_t, r^{i_0}_t)$ is used to denote the pair of the solution process and the switching process at time $t$ that starts from $(x_0, i_0)$ at time $0$. 
\par
For any $ p \in (0, 1) $, define a metric on $ \mathbb{R}^n \times \mathbb{S} $ by  
\begin{align*}
d_p((x, i), (y, j)) = |x - y|^p + \mathbb{I}_{\{i \neq j\}}, \quad (x, i),~(y, j) \in \mathbb{R}^n \times \mathbb{S}.
\end{align*}

The family of all probability measures on $ \mathbb{R}^n \times \mathbb{S} $ is denoted by $ \mathcal{P}(\mathbb{R}^n \times \mathbb{S}) $. The family of all Borel sets in $\mathbb{R}^n$ is denoted by $\mathcal{B}(\mathbb{R}^n)$. The $p$-Wasserstein distance between $u, v \in \mathcal{P}(\mathbb{R}^n \times \mathbb{S})$ for any $p \in (0, 1)$ is defined by
$$
\mathbb{W}_p(u, v) = \inf_{\pi \in \mathcal{C}(u, v)} \int_{(\mathbb{R}^n \times \mathbb{S}) \times ( \mathbb{R}^n \times \mathbb{S})} d_p(X, Y)   \nu (dX, dY),
$$
where $\mathcal{C}(u, v)$ denotes the set of all couplings of $u$ and $v$. The transition probability kernel of $(X^{x_0,i_0}_t, r^{i_0}_t)$ is denoted by $\mathbf{P}_t((x_0, i_0), dx \times \{i\})$, which is a time homogeneous Markov process. Sometimes, we also use $\delta_{(x_0,i_0)}\mathbf{P}_t$ to represents the transition probability kernel. A probability measure $\pi \in \mathcal{P}(\mathbb{R}^n \times \mathbb{S})$ is called an invariant measure of $(X^{x_0,i_0}_t, r^{i_0}_t)$ if
$$
\pi(A \times \{j\}) = \sum_{i=1}^{N} \int_{\mathbb{R}^n} \mathbf{P}_t((x, i), A \times \{j\}) \pi(\mathrm{d}x \times \{i\}) 
\quad \forall t \geq 0,\ A \in \mathcal{B}(\mathbb{R}^n),\ j \in \mathbb{S},
$$
holds. 
\par



To construct the BEM method for \eqref{theSDE}, we need to discretize $r(t)$ firstly. As a classical result, we can generate the one-step transition probability matrix with the step-size $\Delta$ by $(P^{\Delta}_{ij})_{N\times N}=\exp(\Gamma\Delta)$. Then starting from the initial datum $r_0=i_0$, with the help of $(P^{\Delta}_{ij})_{N\times N}$ we could simulate a path $\{r_j\}_{j\in \mathbb{N}_+}$ of $r(t)$ by following the classical way to simulate sample paths of discrete time Markov chain. We refer the readers to page 112 of \cite{MaoYuan2006Book} for the detailed formulation of this process. 
\par
Now the BEM method for the SDE with Markovian switching \eqref{theSDE} is defined as
\begin{equation}
X_{k + 1} = X_k + f(X_{k + 1}, r_{k+1})\Delta + g(X_k, r_k) \Delta B_k, \quad k \geq 0
\label{eq:2}
\end{equation}
with initial data $X_0 = x_0$ and $r_0 = i_0$, where $\Delta>0$ is the constant time step and $\Delta B_k$ is the Brownian motion increment calculated by $B((k + 1)\Delta) - B(k\Delta)$. 

\begin{remark}
    It should be mentioned that in \eqref{eq:2} we use $r_{k+1}$ instead of $r_k$ in the drift part, which is different from most existing works. This setting leads to the fact that $X_{k+1}$ is dependent on $r_{k+1}$. So in our proofs in Section \ref{sec:MRs}, we need to rely on techniques of the conditional expectation quite heavily.
\end{remark}
Hence, we define two filtrations
$$\mathcal{G}_k:= \sigma \bigg(\{r(i\Delta)\}_{i \geq 0}, \{B(j\Delta)\}_{0\leq j\leq k}\bigg) \quad \text{and} \quad \mathcal{F}^r_k:=\sigma \bigg(\{r(j\Delta)\}_{0\leq j\leq k}\bigg),$$ 
which are important in proofs in the next section. Briefly speaking, $\mathcal{G}_k$ contains the information of the whole Markov process but the information of the Brownian motion up to the $k$th step. While, $\mathcal{F}^r_k$ only contains the information of the Markov process up to the the $k$th step.      

\begin{lemma}\label{lemma:WD}
Let Assumption \ref{KhCon} hold. If $\Delta \in (0, 1/(n_M+2))$, then the BEM method \eqref{eq:2} is well defined.
\end{lemma}
\begin{proof}
For any $k \geq 0$, rewrite \eqref{eq:2} into
\begin{align*}
X_{k + 1} - f(X_{k + 1}, r_{k+1})\Delta =X_k + g(X_k, r_k) \Delta B_k.
\end{align*}
Define $G_i(u) = u - f(u,i)\Delta$ for $u\in \mathbb{R}^n, i\in \mathbb{S}$. Due to Assumption \ref{KhCon}, we have 
\begin{align*}
2\langle u_1 - u_2, f(u_1, i) - f(u_2, i)\rangle \leq n_i\lvert u_1 - u_2\rvert^2,
\end{align*}
for all $u_1, u_2\in \mathbb{R}^n, i \in \mathbb{S}$. Then we can get
\begin{align*}
\langle u_1 - u_2, G_i(u_1) - G_i(u_2)\rangle \geq \left(1-\frac{n_i \Delta}{2}\right)\lvert u_1 - u_2\rvert^2.
\end{align*}
By choosing $\Delta \in (0, 1/(n_M+2))$, we have $1-n_i\Delta/2 > 0$. So $G_i(\cdot)$ is monotonic. Then $G_i(\cdot)$ has a unique inverse function $G_i^{-1}(\cdot): \mathbb{R}^n \to \mathbb{R}^n$ such that for any $k \geq 0$
\begin{align*}
X_{k + 1} = G_{r_{k+1}}^{-1}\left(X_k + g(X_k,r_k)\Delta B_k\right).
\end{align*}
Therefore, for any $k \geq 0$ the unique $X_{k + 1}$ can be found provided that $X_k + g(X_k,r_k)\Delta B_k$ and $r_{k+1}$ are given. The proof is completed.
\end{proof}
For the rest of the paper, the step size $\Delta$ is always assumed to satisfy $\Delta \in (0, 1/(n_M+2))$. For convenience, $C$ is regarded as a generic positive constant, whose value may be different from line to line, but is independent of $k$ and $\Delta$. 
\begin{lemma}
$\{(X_k,r_k)\}_{k \geq 0}$ is a homogeneous Markov process with transition probability kernel
\begin{equation*}
\mathbf{P}^{\Delta}_{k}((x, i), A \times \{j\}):= \mathbb{P} \left( (X_k,r_k)\in A \times \{j\} \mid (X_0,r_0) = (x,i)  \right)
\end{equation*}
for $A \in \mathcal{B}(\mathbb{R}^n)$ and $ j \in \mathbb{S}$.
\end{lemma}
\begin{proof}
For any $k \geq 0$, denote 
$$ Z_k := (X_k, r_k) \quad \text{and} \quad  \mathcal{F}_k := \sigma \left(\{r(j\Delta)\}_{0\leq j\leq k}, \{B(j\Delta)\}_{0\leq j\leq k}\right). $$
For any $A \in \mathcal{B}(\mathbb{R}^n)$ and $j \in  \mathbb{S}$, by the product rule we have
\begin{align}
\mathbb{P}(Z_{k+1} \in A\times \{j\} \mid \mathcal{F}_k) 
&= \mathbb{P}(X_{k+1} \in A, r_{k+1} = j \mid \mathcal{F}_k)  \notag \\
&= \mathbb{P}(r_{k+1} = j \mid \mathcal{F}_k) \times \mathbb{P}(X_{k+1} \in A \mid r_{k+1} = j, \mathcal{F}_k).
\label{eq:pr}
\end{align}
As $\{r_k\}_{k  \in \mathbb{N}_+}$ is a Markov process as defined, we have
\begin{equation}\label{eq:r}
\mathbb{P}(r_{k+1} = j \mid \mathcal{F}_k) = \mathbb{P}(r_{k+1} = j \mid Z_k).
\end{equation}
From the proof of Lemma \ref{lemma:WD}, the BEM method \eqref{eq:2} can be rewritten as
\begin{align*}
X_{k + 1} = G_{r_{k+1}}^{-1}\left(X_k + g(X_k,r_k)\Delta B_k\right).
\end{align*}
Since $\Delta B_k= B((k+1) \Delta) - B(k \Delta)$, $\Delta B_k$ is independent of $\mathcal{F}_k$. In addition, $Z_k=(X_k,r_k)$ is $\mathcal{F}_k$-measurable. By a classical result (see for example, Lemma 9.2 at page 87 of \cite{Mao2008book}), we can have 
\begin{align}
\mathbb{P}(X_{k+1} \in A \mid r_{k+1} = j, \mathcal{F}_k) 
&= \mathbb{P}(G_{j}^{-1}\left(X_k + g(X_k,r_k)\Delta B_k\right)  \in A \mid \mathcal{F}_k) \notag \\
&= \mathbb{E} \left(\mathbb{I}_{A} \left(  G_{j}^{-1}\left(X_k + g(X_k,r_k)\Delta B_k\right)\right) \mid \mathcal{F}_k \right) \notag \\
&= \mathbb{E} \left(\mathbb{I}_{A} \left(  G_{j}^{-1}\left(x + g(x,i)\Delta B_k\right)\right) \right)_{x=X_k,i=r_k} \notag \\
&= \mathbb{P} \left(  G_{j}^{-1}\left(x + g(x,i)\Delta B_k  \right) \in A \right)_{x=X_k,i=r_k} \notag \\
&= \mathbb{P} \left(X_{k+1} \in A   \mid  r_{k+1} = j, Z_k \right).
\label{eq:X}
\end{align}
Substituting \eqref{eq:r} and \eqref{eq:X} into \eqref{eq:pr}, we have
\begin{align*}
&~~~~ \mathbb{P}(Z_{k+1} \in A\times \{j\} \mid \mathcal{F}_k) \\
&= \mathbb{P}(r_{k+1} = j \mid Z_k) \times \mathbb{P}\left( X_{k+1} \in A \mid  r_{k+1} = j, Z_k \right) \\
&= \mathbb{P} (Z_{k+1} \in A\times \{j\} \mid Z_k).
\end{align*}
The proof is completed.
\end{proof}

A probability measure $\pi^{\Delta} \in \mathcal{P}(\mathbb{R}^n \times \mathbb{S})$ is called an invariant measure of $(X^{x_0,i_0}_k, r^{i_0}_k)$ if
$$
\pi^{\Delta}(A \times \{j\}) = \sum_{i=1}^{N} \int_{\mathbb{R}^n} \mathbf{P}^{\Delta}_{k}((x, i), A \times \{j\}) \pi^{\Delta}(\mathrm{d}x \times \{i\}) 
\quad \forall k \geq 0,\ A \in \mathcal{B}(\mathbb{R}^n),\ j \in \mathbb{S},
$$
holds. Sometimes we denote $\mathbf{P}^{\Delta}_{k}((x_0, i_0), dx \times \{i\})$ by $\delta_{(x_0,i_0)}\mathbf{P}^{\Delta}_{k}$.

Before the end of this section, we highlight the following equality, which will be employed frequently in proofs in Section \ref{sec:MRs}. For any $a, b \in \mathbb{R}^n$, the equality
\begin{equation} \label{eq:3}
\lvert b \rvert^2 - \lvert a \rvert^2 + \lvert b - a \rvert^2 = 2 \langle b - a, b \rangle
\end{equation}
holds.

\section{Main results}\label{sec:MRs}
We divide this section into two parts. In Section \ref{subsec:nim}, the existence and uniqueness of the invariant measure of the numerical solution generated by the BEM method is proved. In Section \ref{subsec:con}, the convergence of the numerical invariant measure to the underlying counterpart is discussed.
 
\subsection{Existence and uniqueness of numerical invariant measure}\label{subsec:nim}

To obtain the existence and uniqueness of the invariant measure, we need to prepare two lemmas.
\par
Lemma \ref{lem:1} below states the numerical solution is uniformly bounded by some constant that is independent of the time variable.
\begin{lemma}\label{lem:1}
Let Assumptions \ref{KhCon} and \ref{invCon} hold. For any $k \geq 0$ the numerical solution \eqref{eq:2} satisfies 
\begin{align}
\mathbb{E}(\lvert X_k\rvert^2)\leq C \left( \lvert x_0 \rvert^2 +  \lvert g(x_0, i_0) \rvert^2 \right).
\label{eq:2.1}
\end{align}
\end{lemma}

\begin{proof}Due to equality \eqref{eq:3}, we have that
\begin{equation}
\lvert X_k\rvert^2 - \lvert X_{k - 1}\rvert^2 + \lvert X_k - X_{k - 1}\rvert^2 = 2\langle X_k - X_{k - 1}, X_k\rangle.
\label{eq:4}
\end{equation}
From \eqref{eq:2}, we have
\begin{equation*}
2\langle X_k - X_{k - 1}, X_k\rangle= 2\langle f(X_k, r_{k}), X_k\rangle  \Delta + 2\langle g(X_{k - 1}, r_{k - 1})\Delta B_{k - 1}, X_k\rangle.
\end{equation*}
Taking the conditional expectations with respect to $\mathcal{G}_{k-1}$ (defined in Section \ref{sec:MathPre}) on both sides of \eqref{eq:4} and  applying \eqref{Khconext}, we get
\begin{align*}
&\quad \mathbb{E}\left(\lvert X_k \rvert^2 \mid \mathcal{G}_{k-1} \right) - \lvert X_{k - 1}\rvert^2 + \mathbb{E} \left( \lvert X_k - X_{k - 1} \rvert^2  \mid \mathcal{G}_{k-1} \right)\\
&\leq \Delta\left(m + (n_{r_{k}} + 1)\mathbb{E}\left(\lvert X_k\rvert^2 \mid \mathcal{G}_{k-1} \right) \right) - l_2 \Delta\mathbb{E}\left(\lvert g(X_k, r_k)\rvert^2 \mid \mathcal{G}_{k-1} \right) \\
&\quad+ \Delta \lvert g(X_{k - 1}, r_{k - 1})\rvert^2 + \mathbb{E} \left( \lvert X_k - X_{k - 1} \rvert^2 \mid \mathcal{G}_{k-1} \right),
\end{align*}
where $\mathbb{E} \left(\langle g(X_{k - 1}, r_{k - 1})\Delta B_{k - 1}, X_{k-1}\rangle  \mid \mathcal{G}_{k-1} \right) = 0$ is used.

Then, we tidy up the inequality to obtain 
\begin{align*}
&\quad \left(1 -(n_{r_{k}} + 1) \Delta\right)\mathbb{E}\left(\lvert X_k\rvert^2 \mid \mathcal{G}_{k-1} \right) + l_2 \Delta\mathbb{E}\left(\lvert g(X_k, r_k)\rvert^2 \mid \mathcal{G}_{k-1} \right)\\
&\leq \lvert X_{k - 1}\rvert^2 + \Delta \lvert g(X_{k - 1}, r_{k - 1})\rvert^2 + m\Delta.
\end{align*}
Since $\Delta < 1/(n_M+2)$, we have $0 < 1 - (n_i +1)\Delta \leq l_2$ for any $i \in \mathbb{S}$. Then, we have
\begin{align}
&\quad \mathbb{E} \left(\lvert X_k\rvert^2 + \Delta \lvert g(X_k, r_k)\rvert^2  \mid \mathcal{G}_{k-1} \right)  \notag \\
&\leq \frac{1}{1 - (n_{r_{k}} + 1)\Delta}\left(\lvert X_{k - 1}\rvert^2 + \Delta \lvert g(X_{k - 1}, r_{k - 1})\rvert^2\right) + \frac{m\Delta}{1 - (n_{r_{k}} + 1)\Delta}. 
\label{eq:1stitr}
\end{align}
Taking the conditional expectations with respect to $\mathcal{G}_{k-2}$ on both sides of \eqref{eq:1stitr}, we have
\begin{align}
&\quad \mathbb{E} \left(\lvert X_k\rvert^2 + \Delta \lvert g(X_k, r_k)\rvert^2  \mid \mathcal{G}_{k-2} \right)  \notag \\
&\leq \frac{1}{1 - (n_{r_{k}} + 1)\Delta}\mathbb{E} \left(\lvert X_{k - 1}\rvert^2 + \Delta \lvert g(X_{k - 1}, r_{k - 1})\rvert^2 \mid \mathcal{G}_{k-2} \right) + \frac{m\Delta}{1 - (n_{r_{k}} + 1)\Delta}. 
\label{eq:2nditr}
\end{align}
Applying \eqref{eq:1stitr} to the right hand side of \eqref{eq:2nditr}, we have
\begin{align*}
&\quad \mathbb{E} \left(\lvert X_k\rvert^2 + \Delta \lvert g(X_k, r_k)\rvert^2  \mid \mathcal{G}_{k-2} \right)  \\
&\leq \frac{1}{1 - (n_{r_{k}} + 1)\Delta} \times \frac{1}{1 - (n_{r_{k-1}} + 1)\Delta}\left(\lvert X_{k - 2}\rvert^2 + \Delta \lvert g(X_{k - 2}, r_{k - 2})\rvert^2\right) \\
&\quad+ \frac{m\Delta}{1 - (n_{r_{k}} + 1)\Delta}+\frac{1}{1 - (n_{r_{k}} + 1)\Delta} \times \frac{m\Delta}{1 - (n_{r_{k-1}} + 1)\Delta}.
\end{align*}
Repeating the preceding processes leads to
\begin{align}
&\quad \mathbb{E} \left(\lvert X_k\rvert^2 + \Delta \lvert g(X_k, r_k)\rvert^2  \mid \mathcal{G}_{0} \right) \notag \\
&\leq \prod_{j = 1}^{k}\frac{1}{1 -(n_{r_j} + 1) \Delta}\left(\lvert x_0\rvert^2 + \Delta\lvert g(x_0, i_0)\rvert^2\right) + m\Delta\sum_{i = 1}^{k}\prod_{j = k-i+1}^{k}\frac{1}{1 - (n_{r_j} + 1)\Delta}.
\label{eq:enditr}
\end{align}
Taking expectations on both sides of \eqref{eq:enditr} and using the homogeneous property of the Markov chain for the second term on the right hand side of \eqref{eq:enditr}, for any $k \geq 0$ we obtain
\begin{align}
\mathbb{E}(\lvert X_k\rvert^2)&\leq \left(\lvert x_0\rvert^2 + \Delta\lvert g(x_0, i_0)\rvert^2\right)\mathbb{E}\left(\prod_{j = 1}^{k}\frac{1}{1 -(n_{r_j} + 1) \Delta}\right)\notag\\
&\quad+ m\Delta\sum_{i = 1}^{k}\mathbb{E}\left[\mathbb{E}\left(\prod_{j = k-i+1}^{k}\left(\frac{1}{1 - (n_{r_j} + 1)\Delta}\right)\mid \mathcal{F}^r_{k-i}\right)\right]\notag\\
&\leq \left(\lvert x_0\rvert^2 + \Delta\lvert g(x_0, i_0)\rvert^2\right)\mathbb{E}\left(\prod_{j = 1}^{k}\frac{1}{1 -(n_{r_j} + 1) \Delta}\right)\notag \\
&\quad+ m\Delta\sum_{i = 1}^{k}\mathbb{E}\left(\prod_{j = 1}^{i} \frac{1}{1 - (n_{r_j} + 1)\Delta}\right) \notag\\
&\leq \left(\lvert x_0\rvert^2 + \Delta\lvert g(x_0, i_0)\rvert^2\right)\mathbb{E}\left[\exp\left(\sum_{j = 1}^{k}\log\left(\frac{1}{1 - (n_{r_j} + 1)\Delta}\right)\right)\right]\notag \\
&\quad+ m\Delta\sum_{i = 1}^{k}\mathbb{E}\left[\exp\left(\sum_{j =  1}^{i}\log\left(\frac{1}{1 - (n_{r_j} + 1)\Delta}\right)\right)\right]\notag\\
&=: H_1 + H_2.
\label{eq:7}
\end{align}
Then, by using the ergodic property of the Markov chain and the inequality $\log(1 + u) \leq u~ \text{for any}~ u > -1$, we have
\begin{align}\label{eq:8}
\lim_{i \to \infty} \frac{1}{i} \sum_{j = 1}^{i} \log \left( \frac{1}{1 - (n_{r_j} + 1)\Delta} \right) &= \sum_{j \in \mathbb{S}}^{} \mu_j \log \left( \frac{1}{1 - (n_j + 1)\Delta} \right) \notag\\
&= \sum_{{j \in \mathbb{S}}}^{} \mu_j \log \left( 1 + \frac{(n_j + 1)\Delta}{1 - (n_j + 1)\Delta} \right) \notag\\
&\leq \Delta \sum_{j \in \mathbb{S}}^{} \mu_j \frac{n_j + 1}{1 - (n_j + 1)\Delta}\notag\\
&\leq \Delta \sum_{{j \in \mathbb{S}}}^{} \mu_j \frac{n_j+1}{1 - (n_j+1)/(n_M + 2)} = -\lambda_1 \Delta~~~\text{a.s.},
\end{align}
%
%
%
%
%
where the last equality is derived from Assumption \ref{invCon}. From \eqref{eq:8}, it is clear that 
\begin{equation*}
\lim_{i \to \infty} \exp \left( \frac{\lambda_1 \Delta}{2} i + \sum_{j = 1}^{i} \log \left( \frac{1}{1 - (n_{r_j} + 1)\Delta} \right) \right) = 0 ~~~\text{a.s.} 
\end{equation*}

By the Fatou lemma, we have
\begin{align*}
\limsup_{i \to \infty} \mathbb{E} \left( \exp \left( \frac{\lambda_1 \Delta}{2} i + \sum_{j = 1}^{i} \log \left( \frac{1}{1 - (n_{r_j} + 1)\Delta }\right) \right) \right) = 0.
\end{align*}
Thus for any $\epsilon > 0$, there exists a positive integer $N$ such that for any $i > N$, 
\begin{align*}
\mathbb{E} \left( \exp \left( \frac{\lambda_1 \Delta}{2} i + \sum_{j = 1}^{i} \log \left( \frac{1}{1 - (n_{r_j} + 1)\Delta }\right) \right) \right) \leq \epsilon.
\end{align*}
Set $\epsilon =\exp(- \frac{\lambda_1 \Delta}{2} i )$, we have 
\begin{align*}
\mathbb{E} \left( \exp \left( \sum_{j = 1}^{i} \log \left( \frac{1}{1 - (n_{r_j} + 1)\Delta} \right) \right) \right) \leq \exp \left( - \lambda_1 \Delta i \right).
\end{align*}
Therefore, for any $k >N + 1$ we have
\begin{align}
H_1 \leq \left( \lvert x_0\rvert^2 + \Delta\lvert g(x_0, i_0)\rvert^2 \right)\exp \left( - \lambda_1 \Delta k \right).
\label{eq:9}
\end{align}
For the given $N$, we have
\begin{align}
H_2&= m \Delta\sum_{i = 1}^{k} \mathbb{E} \left[ \exp \left( \sum_{j = 1}^{i} \log \left( \frac{1}{1 - (n_{r_j} + 1)\Delta} \right) \right) \right]\notag \\
&\leq m \Delta\sum_{i = 1}^{N} \left( \frac{1}{1 -  (n_M + 1)\Delta} \right)^i + m\Delta \sum_{i = N + 1}^{k} \exp \left( - \lambda_1 \Delta i \right) \notag\\
&\leq C\Delta + m\Delta \sum_{i = N + 1}^{k} \exp \left( - \lambda_1 \Delta i \right) \quad \quad\quad\quad\quad\quad\forall k>N.
\label{eq:10}
\end{align}
Substituting \eqref{eq:9} and \eqref{eq:10} into \eqref{eq:7} gives
\begin{align*}
\mathbb{E}(\lvert X_k\rvert^2) &\leq C\left( \lvert x_0 \rvert^2 + \Delta \lvert g(x_0, i_0) \rvert^2 \right) \exp \left( - \lambda_1 \Delta (k \vee (N + 1)) \right) \\
&\quad + C\Delta + m\Delta \sum_{i = N + 1}^{k \vee (N + 1)} \exp \left( - \lambda_1 \Delta i \right) \\
&\leq C \left( \lvert x_0 \rvert^2 + \lvert g(x_0, i_0) \rvert^2 \right)\quad \quad\quad\quad\quad\quad\quad\quad\forall k \geq 0.
\end{align*}
Therefore, the assertion follows.
\end{proof}

Next, we present our second lemma below, which states some attractive property of numerical solutions starting from different initial data $(x_0,i_0)$ and $(y_0,j_0)$. 

\begin{lemma}\label{lem:2}
Let Assumptions \ref{PolyCon}, \ref{KhCon} and \ref{invCon} hold. For any $p \in (0, 4/q)$, the numerical solution \eqref{eq:2} satisfies
\begin{align}
\mathbb{E} \left| X_{k}^{x_0,i_0} - X_{k}^{y_0,j_0} \right|^p \leq C e^{- \gamma k \Delta},
\label{eq:2.3}
\end{align}
where $C$ and $\gamma$ are postive constants independent of $k$ and $\Delta$.

\end{lemma}

\begin{proof}
We divide the proof into two steps. In the first step, we look at the mean square difference between $X_k^{x_0,i_0}$ and $X_k^{y_0,i_0}$, which are numerical solutions starting from different initial states, $x_0$ and $y_0$, but the same $i_0$. In the second step, we further look at our final target that initial values are totally different.
\medskip
\par \noindent
{\bf Step 1}
\par \noindent
Denote $D_k^{i_0} = X_k^{x_0,i_0} - X_k^{y_0,i_0}$. By Assumption \ref{KhCon}, we have that
\begin{align*}
&\quad 2 \mathbb{E}  \left( \left\langle D_k^{i_0} - D_{k-1}^{i_0}, D_k^{i_0} \right\rangle \mid \mathcal{G}_{k-1} \right) \\
&= 2 \mathbb{E} \left( \left\langle \left(f\left(X_k^{x_0,i_0}, r_{k}^{i_0}\right) - f\left(X_k^{y_0,i_0}, r_{k}^{i_0}\right)\right)\Delta, D_k^{i_0} \right\rangle \mid \mathcal{G}_{k-1} \right) \\
&\quad+ 2 \mathbb{E} \left( \left\langle \left(g\left(X_{k-1}^{x_0,i_0}, r_{k-1}^{i_0}\right) - g\left(X_{k-1}^{y_0,i_0}, r_{k-1}^{i_0}\right)\right) \Delta B_{k-1}, D_k^{i_0} \right\rangle \mid \mathcal{G}_{k-1} \right)\\
&\leq \Delta \left( n_{r_k^{i_0}} \mathbb{E}\left( \lvert D_k^{i_0} \lvert^2 \mid \mathcal{G}_{k-1} \right) - l_1 \mathbb{E}\left( \left\lvert g\left(X_k^{x_0,i_0}, r_{k}^{i_0}\right) - g\left(X_k^{y_0,i_0}, r_{k}^{i_0}\right) \right\lvert^2 \right) \mid \mathcal{G}_{k-1} \right)  \\
&\quad+2 \mathbb{E} \left( \left\langle \left(g\left(X_{k-1}^{x_0,i_0}, r_{k-1}^{i_0}\right) - g\left(X_{k-1}^{y_0,i_0}, r_{k-1}^{i_0}\right)\right) \Delta B_{k-1}, D_k^{i_0} -D_{k-1}^{i_0}\right\rangle \mid \mathcal{G}_{k-1} \right),
\end{align*}
where
\begin{equation*}
\mathbb{E} \left( \left\langle \left(g\left(X_{k-1}^{x_0,i_0}, r_{k-1}^{i_0}\right) - g\left(X_{k-1}^{y_0,i_0}, r_{k-1}^{i_0}\right)\right) \Delta B_{k-1}, D_{k-1}^{i_0}\right\rangle \mid \mathcal{G}_{k-1} \right) = 0
\end{equation*}
is used and $\mathcal{G}_{k}$ is defined in Section \ref{sec:MathPre}.

Thanks to \eqref{eq:3}, we have
\begin{align*}
&~~~~\mathbb{E}\left(\rvert D_k^{i_0} \rvert^2 \mid \mathcal{G}_{k-1} \right) - \mathbb{E}\left(\rvert D_{k-1}^{i_0} \rvert^2 \mid \mathcal{G}_{k-1} \right) + \mathbb{E}\left(\rvert D_k^{i_0} - D_{k-1}^{i_0} \rvert^2 \mid \mathcal{G}_{k-1} \right)\\ 
&= 2 \mathbb{E} \left( \left\langle D_k^{i_0} - D_{k-1}^{i_0}, D_k^{i_0} \right\rangle \mid \mathcal{G}_{k-1} \right)\\
&\leq n_{r_{k}^{i_0}} \Delta\mathbb{E}\left( \lvert D_k^{i_0} \lvert^2 \mid \mathcal{G}_{k-1} \right) -  l_1 \Delta\mathbb{E}\left( \left\lvert g\left(X_k^{x_0,i_0}, r_{k}^{i_0}\right) - g\left(X_k^{y_0,i_0}, r_{k}^{i_0}\right) \right\lvert^2 \mid \mathcal{G}_{k-1} \right) \\
&\quad+ \Delta \mathbb{E}\left( \left\rvert g\left(X_{k-1}^{x_0,i_0}, r_{k-1}^{i_0}\right) - g\left(X_{k-1}^{y_0,i_0}, r_{k-1}^{i_0}\right) \right\rvert^2 \mid \mathcal{G}_{k-1} \right) + \mathbb{E}\left( \rvert D_k^{i_0,j_0} - D_{k-1}^{i_0,j_0} \rvert^2 \mid \mathcal{G}_{k-1} \right).
\end{align*}
Rearranging terms on both sides, we can get
\begin{align*}
&~~~~\left(1 - n_{r_{k}^{i_0}}\Delta \right) \mathbb{E} \left( \rvert D_k^{i_0} \rvert^2 \mid \mathcal{G}_{k-1} \right) +  l_1\Delta \mathbb{E} \left( \left\rvert g\left(X_k^{x_0,i_0}, r_{k}^{i_0}\right) - g\left(X_k^{y_0,i_0}, r_{k}^{i_0}\right) \right\rvert^2 \mid \mathcal{G}_{k-1} \right) \\
&\quad \leq \mathbb{E} \left( \rvert D_{k-1}^{i_0} \rvert^2 \mid \mathcal{G}_{k-1} \right) + \Delta \mathbb{E} \left( \left\rvert g\left(X_{k-1}^{x_0,i_0}, r_{k-1}^{i_0}\right) - g\left(X_{k-1}^{y_0,i_0}, r_{k-1}^{i_0}\right) \right\rvert^2 \mid \mathcal{G}_{k-1} \right) .
\end{align*}
Due to $\Delta < 1/(n_M+2)$, we have $ 0 < 1 -  n_{r_j^{i_0}} \Delta\leq l_1 ~\text{for any}~ i \in \mathbb{S}$. Then we have 
\begin{align}
&~~~~\mathbb{E} \left( \rvert D_k^{i_0} \rvert^2 + \Delta  \left\rvert g\left(X_k^{x_0,i_0}, r_{k}^{i_0}\right) - g\left(X_k^{y_0,i_0}, r_{k}^{i_0}\right) \right\rvert^2  \mid \mathcal{G}_{k-1} \right)\notag \\
&\quad \leq \frac{1}{1 -  n_{r_{k}^{i_0}}\Delta}  \left( \rvert D_{k-1}^{i_0} \rvert^2  + \Delta \left\rvert g\left(X_{k-1}^{x_0,i_0}, r_{k-1}^{i_0}\right) - g\left(X_{k-1}^{y_0,i_0}, r_{k-1}^{i_0}\right) \right\rvert^2  \right) .
\label{eq:20}
\end{align}
Applying the similar trick used in Lemma \ref{lem:1} that taking conditional expectations $\mathcal{G}_{k-2}$, $\mathcal{G}_{k-3}$, $\cdots$, $\mathcal{G}_{0}$ sequentially and using \eqref{eq:20} repeatedly, we can achieve
\begin{align}
&\quad \mathbb{E} \left( \rvert D_k^{i_0} \rvert^2 \mid \mathcal{G}_{0}  \right) \notag\\
&\leq \left( \rvert x_0 - y_0 \rvert^2 + \Delta \left\rvert g(x_0, i_0) - g(y_0, i_0) \right\rvert^2 \right) \prod_{j=1}^{k} \frac{1}{1 -  n_{r_j^{i_0}}\Delta} \notag\\
& = \left( \rvert x_0 - y_0 \rvert^2 + \Delta \left\rvert g(x_0, i_0) - g(y_0, i_0) \right\rvert^2 \right)  \left( \exp \left( \sum_{j=1}^{k} \log \frac{1}{1 -  n_{r_j^{i_0}}\Delta} \right) \right).
\label{eq:21}
\end{align}
The ergodic property of the Markov chain, the inequality $\log(1 + u) \leq u ~\text{for all} ~ u > -1$ and Assumption \ref{invCon} indicate
\begin{align*}
\lim_{i \to \infty} \frac{1}{i} \sum_{j=1}^{i} \log \left( \frac{1}{1 -  n_{r_j^{i_0}}\Delta} \right) &= \sum_{j \in S} \mu_j \log \frac{1}{1 -  n_j\Delta} \\
&= \sum_{j \in S} \mu_j \log \left( 1 + \frac{ n_j\Delta}{1 -  n_j\Delta} \right) \\
& \leq \Delta \sum_{j \in S} \mu_j \frac{ n_j\Delta}{1 -  n_j\Delta} \\
&\leq \Delta \sum_{j \in S} \mu_j \frac{n_j+1}{1 - (n_j+1)/(n_M+2)} = -  \lambda_2 \Delta \quad \text{a.s.},
\end{align*}
which then indicates
\begin{align*}
\lim_{i \to \infty} \exp \left( \frac{ \lambda_2 \Delta}{2} i + \sum_{j=1}^{i} \log \left( \frac{1}{1 -  n_{r_j^{i_0}}\Delta} \right) \right) = 0 \quad \text{a.s.}
\end{align*}
By the Fatou lemma, we have
\begin{align*}
\limsup_{i \to \infty} \mathbb{E} \left[ \exp \left( \frac{ \lambda_2 \Delta}{2} i + \sum_{j=1}^{i} \log \left( \frac{1}{1 -  n_{r_j^{i_0}}\Delta} \right) \right) \right] = 0.
\end{align*}
Thus for any $\epsilon > 0$, there exists a positive integer $N$ such that for all $i > N$, 
\begin{align*}
\mathbb{E} \left[ \exp \left( \frac{\lambda_2 \Delta}{2} i + \sum_{j = 0}^{i-1} \log \left( \frac{1}{1 -  n_{r_j^{i_0}} \Delta}\right) \right) \right] \leq \epsilon.
\end{align*}
Set $\epsilon =\exp(- \frac{\lambda_2 \Delta}{2} i )$, we have
\begin{align}\label{eq:Eexp}
\mathbb{E} \left[ \exp \left( \sum_{j = 0}^{i-1} \log \left( \frac{1}{1 -  n_{r_j^{i_0}}\Delta} \right) \right) \right] \leq \exp \left( - \lambda_2 \Delta i \right).
\end{align}
Taking expectations on both sides of \eqref{eq:21} and substituting \eqref{eq:Eexp} into it, for any $k > N$ we obtain
\begin{align}
&\mathbb{E} \left( \left\rvert X_k^{x_0,i_0} - X_k^{y_0,i_0} \right\rvert^2 \right)\leq \left( \rvert x_0 - y_0 \rvert^2 + \left\rvert g(x_0, i_0) - g(y_0, i_0) \right\rvert^2 \right)\exp(-  \lambda_2\Delta k). 
\label{eq:12}
\end{align} 
\medskip
\par \noindent
{\bf Step 2}
\par \noindent
Define $\tau = \inf\{k \geq 0: r^{i_0}_k = r^{j_0}_k\}$. Since the state space $\mathbb{S}$ is finite, and $\Gamma$ is irreducible,
there exists $\theta > 0$ such that
\begin{align}
\mathbb{P}(\tau > k) \leq e^{-\theta k \Delta}
\label{eq:2.2}
\end{align}
for any $k > 0$. 
For any $p \in (0,2)$, it is clear that
\begin{align}\label{eq:2parts}
&~~~~ \mathbb{E} \left| X_{k}^{x_0,i_0} - X_{k}^{y_0,j_0} \right|^p \notag \\
&= \mathbb{E} \left( \left| X_{k}^{x_0,i_0} - X_{k}^{y_0,j_0} \right|^p \mathbb{I}_{\left\{ \tau > \left[ \frac{k}{2} \right]\right\}} \right) + \mathbb{E} \left( \left| X_{k}^{x_0,i_0} - X_{k}^{y_0,j_0} \right|^p \mathbb{I}_{\left\{ \tau \leq \left[ \frac{k}{2} \right] \right\}} \right) \notag \\
&:=M_1 + M_2.
\end{align}
For $M_1$ in \eqref{eq:2parts}, the H\"older inequality together with \eqref{eq:2.2} gives
\begin{align} \label{eq:est1stpart}
&~~~~\mathbb{E} \left( \left| X_{k}^{x_0,i_0} - X_{k}^{y_0,j_0} \right|^p \mathbb{I}_{\left\{ \tau > \left[ \frac{k}{2} \right]\right\}} \right) \notag \\
&\leq \left(\mathbb{E} \left| X_{k}^{x_0,i_0} - X_{k}^{y_0,j_0} \right|^2 \right)^{\frac{p}{2}} \left(\mathbb{P}\left( \tau > \left[ \frac{k}{2} \right] \right)\right)^{1-\frac{p}{2}} \notag \\
&\leq e^{-\frac{2-p}{4}\theta k\Delta} \left(\mathbb{E} \left| X_{k}^{x_0,i_0} - X_{k}^{y_0,j_0} \right|^2\right)^{\frac{p}{2}} \notag \\
&\leq e^{-\frac{2-p}{4}\theta k\Delta} \left( 2\mathbb{E} \left| X_{k}^{x_0,i_0} \right|^2 + 2\mathbb{E} \left|X_{k}^{y_0,j_0} \right|^2\right)^{\frac{p}{2}} \notag \\
&\leq C e^{-\frac{2-p}{4}\theta k\Delta} \left( |x_0|^p + |g(x_0,i_0)|^p+ |y_0|^p + |g(y_0,j_0)|^p \right)
\end{align}
where the elementary inequality $(a+b)^p \leq 2^p (a^p + b^p)$ for any $a,b>0$ and \eqref{eq:2.1} are used for the last two inequalities.
\par
For $M_2$ in \eqref{eq:2parts}, by the Markov property, the fact $r_\tau^{i_0}=r_\tau^{j_0}$ and \eqref{eq:12}, we have
\begin{align}\label{eq:M2twoparts}
&~~~~\mathbb{E} \left( \left| X_{k}^{x_0,i_0} - X_{k}^{y_0,j_0} \right|^p \mathbb{I}_{\left\{ \tau \leq \left[ \frac{k}{2} \right] \right\}} \right) \notag \\
&= \mathbb{E} \left( \mathbb{I}_{\left\{ \tau \leq \left[ \frac{k}{2} \right] \right\}} \mathbb{E} \left( \left| X_{k}^{x_0,i_0} - X_{k}^{y_0,j_0} \right|^p \mid \mathcal{F}_{\tau\Delta} \right) \right) \notag \\
&= \mathbb{E} \left( \mathbb{I}_{\left\{ \tau \leq \left[ \frac{k}{2} \right] \right\}} \mathbb{E}  \left(\left| X_{k - \tau}^{X_{\tau}^{x_0,i_0}, r_\tau^{i_0}} - X_{k - \tau}^{X_{\tau}^{y_0,j_0}, r_\tau^{j_0}} \right|^p\right)  \right) \notag \\
&\leq C e^{-\frac{p}{2} k \lambda_2 \Delta} \mathbb{E} \left( \mathbb{I}_{\left\{ \tau \leq \left[ \frac{k}{2} \right] \right\}}  \left( \mathbb{E} \left| X_{\tau}^{x_0,i_0} - X_{\tau}^{y_0,j_0} \right|^p + \mathbb{E} \left| g \left( X_{\tau}^{x_0,i_0}, r_{\tau}^{i_0} \right) - g \left( X_{\tau}^{y_0,j_0}, r_{\tau}^{j_0} \right) \right|^p \right) \right) \notag \\
&\leq C e^{-\frac{p}{2} k \lambda_2 \Delta} \left( \mathbb{E} \left| X_{\tau \wedge \left[ \frac{k}{2} \right]}^{x_0,i_0} - X_{\tau \wedge \left[ \frac{k}{2} \right]}^{y_0,j_0} \right|^p + \mathbb{E} \left| g \left( X_{\tau \wedge \left[ \frac{k}{2} \right]}^{x_0,i_0}, r_{\tau \wedge \left[ \frac{k}{2} \right]}^{i_0} \right) - g \left( X_{\tau \wedge \left[ \frac{k}{2} \right]}^{y_0,j_0}, r_{\tau \wedge \left[ \frac{k}{2} \right]}^{j_0} \right) \right|^p \right) \notag \\
&:= C e^{-\frac{p}{2} k \lambda_2 \Delta} \left(M_{21}+M_{22}\right).
\end{align}
For $M_{21}$ in \eqref{eq:M2twoparts}, by the elementary inequality we have
\begin{align}\label{eq:M21}
&~~~~ \mathbb{E}  \left| X_{\tau \wedge \left[ \frac{k}{2} \right]}^{x_0,i_0} - X_{\tau \wedge \left[ \frac{k}{2} \right]}^{y_0,j_0} \right|^p \notag\\
&\leq 2^p \left( \mathbb{E} \left| X_{\tau \wedge \left[ \frac{k}{2} \right]}^{x_0,i_0} \right|^p + \mathbb{E} \left| X_{\tau \wedge \left[ \frac{k}{2} \right]}^{y_0,j_0} \right|^p \right) \notag\\
&= 2^p \mathbb{E} \left( \sum_{m=0}^{\left[ \frac{k}{2} \right]} \left| X_{m}^{x_0,i_0} \right|^p \mathbb{I}_{\left\{ \tau \wedge \left[ \frac{k}{2} \right] = m \right\}} (\omega) \right) + 2^p \mathbb{E} \left( \sum_{m=0}^{\left[ \frac{k}{2} \right]} \left| X_{m}^{y_0,j_0} \right|^p \mathbb{I}_{\left\{ \tau \wedge \left[ \frac{k}{2} \right] = m \right\}} (\omega) \right) \notag\\
&\leq 2^p \sum_{m=0}^{\left[ \frac{k}{2} \right]} \left[ \mathbb{E} \left( \left| X_{m}^{x_0,i_0} \right|^p \right) + \mathbb{E} \left( \left| X_{m}^{y_0,j_0} \right|^p \right) \right] \notag\\
&\leq C \left( |x_0|^p + |y_0|^p + | g(x_0, i_0) |^p + | g(y_0, j_0)|^p \right),
\end{align}
where \eqref{eq:2.1} is used to derive the last inequality.

For $M_{22}$ in \eqref{eq:M2twoparts}, in the similar manner we have
\begin{align*}
&~~~~\mathbb{E} \left| g \left( X_{\tau \wedge \left[ \frac{k}{2} \right]}^{x_0,i_0}, r_{\tau \wedge \left[ \frac{k}{2} \right]}^{i_0} \right) - g \left( X_{\tau \wedge \left[ \frac{k}{2} \right]}^{y_0,j_0}, r_{\tau \wedge \left[ \frac{k}{2} \right]}^{j_0} \right) \right|^p \notag \\
&\leq 2^p \sum_{m=0}^{\left[ \frac{k}{2} \right]}\left( \mathbb{E} \left| g \left( X_{m}^{x_0,i_0}, r_{m}^{i_0} \right) \right|^p + \mathbb{E} \left| g \left( X_{m}^{y_0,j_0}, r_{m}^{j_0} \right) \right|^p \right) \notag \\
&\leq 2^p b_M \sum_{m=0}^{\left[ \frac{k}{2} \right]} \left( 2 + \mathbb{E} \left| X_{m}^{x_0,i_0} \right|^{pq/2} + \mathbb{E} \left| X_{m}^{y_0,j_0} \right|^{pq/2} \right),
\end{align*}
where the last inequality is obtained by \eqref{eq:13} and  $b_M := \max_{i \in \mathbb{S}}{|b_i|}$. Since $q\geq 2$, we can choose $p \in (0, 4/q)$ such that $pq/2 \in (0,2)$. Using \eqref{eq:2.1}, we further have
\begin{equation}\label{eq:M22}
    M_{22} \leq C \left(2+ |x_0|^{pq/2} + |g(x_0,i_0)|^{pq/2} + |y_0|^{pq/2} + |g(y_0,j_0)|^{pq/2}  \right).
\end{equation}
Substituting \eqref{eq:M21} and \eqref{eq:M22} into \eqref{eq:M2twoparts}, we obtain
\begin{equation}
    \mathbb{E} \left( \left| X_{k}^{x_0,i_0} - X_{k}^{y_0,j_0} \right|^p \mathbb{I}_{\left\{ \tau \leq \left[ \frac{k}{2} \right] \right\}} \right) \leq C e^{-\frac{p}{2} k \lambda_2 \Delta},
\end{equation}
which together with \eqref{eq:est1stpart} completes this proof.
\end{proof}

Now, we are ready to present the main result on the existence and uniqueness of the numerical invariant measure.
\begin{theorem}\label{theo:1.2}
Let Assumptions \ref{PolyCon}, \ref{KhCon} and \ref{invCon} hold. the numerical solution \eqref{eq:2} converges,  with some exponential rate $\gamma_{\Delta} > 0$, to a unique invariant measure $\pi^{\Delta}$ in the Wasserstein distance.
\end{theorem}
The proof of Theorem \ref{theo:1.2} is quite straightforward. Briefly speaking, for a homogeneous Markov process Lemma \ref{lem:1} guarantees the existence of the numerical invariant measure and Lemma \ref{lem:2} ensures the uniqueness of it. So when Lemmas \ref{lem:1} and \ref{lem:2} are ready to use, the proof of Theorem \ref{theo:1.2} no longer depends on the structure of the numerical method but just follows some standard approaches (see for example, Theorem 3.3 in \cite{BaoShaoYuan2016} or Theorem 3.5 in \cite{LiMaYangYuan2018}). To keep this paper self-contained, we provide the proof as follows.
\begin{proof}
Due to \eqref{eq:2.1} and Chebyshev's inequality, for any initial data $(x_0, i_0)$ we obtain that $\{\delta_{(x_0, i_0)}\mathbf{P}^{\Delta}_{k}\}$ is tight. Then a subsequence can be extracted which converges weakly to an invariant measure  $\pi^{\Delta} \in \mathcal{P}(\mathbb{R}^n \times \mathbb{S})$. From \eqref{eq:2.3} and \eqref{eq:2.2}, we can get 
\begin{align}
W_p\left(\delta_{(x_0,i_0)} \mathbf{P}_{k}^\Delta, \delta_{(y_0,j_0)} \mathbf{P}_{k}^\Delta\right) 
&\leq \mathbb{E}\lvert X_k^{x_0,i_0} - X_k^{y_0,j_0}\rvert^p + \mathbb{P}\left(r_k^{i_0} \neq r_k^{j_0}\right) \notag\\
&\leq K e^{- \gamma k \Delta} + e^{-\theta k \Delta} \notag\\
&\leq K e^{- \gamma_{\Delta} k\Delta},
\label{eq:2.5}
\end{align}
where $\gamma_{\Delta} = \min\{\gamma, \theta\}$. Thanks to Assumption \ref{PolyCon}, Lemma \ref{lem:1} and the Kolmogorov-Chapman equation, for any $k, l >0$ and $p \in (0,1)$, yields that
\begin{align}
&W_p(\delta_{(x_0,i_0)}\mathbf{P}_{k}^\Delta, \delta_{(x_0,i_0)}\mathbf{P}_{k+l}^\Delta) 
= W_p(\delta_{(x_0,i_0)}\mathbf{P}_{k}^\Delta, \delta_{(x_0,i_0)} \mathbf{P}_{k}^\Delta\mathbf{P}_{l}^\Delta) \notag \\
&\leq \int_{\mathbb{R}^n \times \mathbb{S}} W_p(\delta_{(x_0,i_0)}\mathbf{P}_{k}^\Delta, \delta_{(y_0,j_0)}\mathbf{P}_{k}^\Delta) \mathbf{P}_{l}^\Delta((x_0, i_0), dy_0 \times \{j\}) \notag \\
&\leq \sum_{j \in \mathbb{S}} \int_{\mathbb{R}^n} Ke^{-\gamma_{\Delta} k\Delta} \mathbf{P}_{l}^\Delta((x_0, i_0), dy_0\times \{j\}) \notag \\
&= C\left(|x_0|^p + \mathbb{E}\lvert X_l^{x_0,i_0} \rvert^p + | g(x_0, i_0) |^p + \mathbb{E}| g(X_l^{x_0,i_0}, r_l^{i_0}) |^p\right)e^{-\gamma_{\Delta} k\Delta} \notag\\
&\leq C e^{-\gamma_{\Delta} k\Delta}.\notag
\end{align}
Then, let $l \to \infty$, we have
\begin{align*}
W_p(\delta_{(x_0,i_0)}\mathbf{P}_{k}^\Delta, \pi^\Delta) 
\to 0, \quad k \to \infty,
\end{align*}
where $ \pi^\Delta $ is the invariant measure of $ \{\delta_{(x_0,i_0)}\mathbf{P}_{k}^\Delta\} $. To investigate the uniqueness of the invariant measure, we assume $ \pi_1^\Delta, \pi_2^\Delta \in \mathcal{P}(\mathbb{R}^n \times \mathbb{S}) $ are respectively the invariant measures of $ (X_k^{x_0,i_0}, r_k^{i_0}) $ and $ (X_k^{y_0,j_0}, r_k^{j_0}) $. Thus we have
\begin{align*}
W_p(\pi_1^\Delta, \pi_2^\Delta) &\leq \int_{(\mathbb{R}^n \times \mathbb{S}) \times (\mathbb{R}^n \times \mathbb{S})} W_p(\delta_{(x_0,i_0)}\mathbf{P}_{k}^\Delta, \delta_{(y_0,j_0)}\mathbf{P}_{k}^\Delta)\pi(dx_0 \times di_0, dy_0 \times dj_0),
\end{align*}
where $ \pi $ is a coupling of $ \pi_1^\Delta $ and $ \pi_2^\Delta $. Due to \eqref{eq:2.5}, the uniqueness of invariant measures can be derived.
\end{proof}

\subsection{Convergence of numerical invariant measure}\label{subsec:con}
The purpose of studying the numerical invariant measure is to employ it to approximate the underlying counterpart. So we first borrow a result from \cite{BaoShaoYuan2016}, which states the existence and uniqueness of the invariant measure for the underlying equation \eqref{theSDE}.

\begin{theorem} \label{theo:1.1}
Let Assumptions \ref{PolyCon}, \ref{KhCon} and \ref{invCon} hold. The solution of \eqref{theSDE} converges in the Wasserstein distance to a unique invariant measure $\pi$ with some exponential rate $\gamma > 0$.
\end{theorem}
Now, we present the result on the convergence of the numerical invariant measure to the underlying counterpart below.

\begin{theorem}\label{theo:1.3}
Let Assumptions \ref{PolyCon}, \ref{KhCon} and \ref{invCon} hold. The numerical invariant measure $\pi^\Delta$ converges  to the underlying one $\pi$ in the Wasserstein distance with the rate $1/2$, i.e.  $W_p\left(\pi, \pi^\Delta\right)  \leq C \Delta^{p/2}$.
\end{theorem}
\begin{proof}
Thanks to the triangle inequality, we have
\begin{align*}
W_p\left(\pi, \pi^\Delta\right) &\leq W_p\left(\delta_{(x_0,i_0)}\mathbf{P}_{k\Delta}, \pi\right) + W_p\left(\delta_{(x_0,i_0)}\mathbf{P}_{k\Delta}, \delta_{(x_0,i_0)}\mathbf{P}_{k}^\Delta\right)+ W_p\left(\delta_{(x_0,i_0)}\mathbf{P}_{k}^\Delta, \pi^\Delta\right).
\end{align*}
Due to Theorems \ref{theo:1.2} and  \ref{theo:1.1}, for any $\Delta \in (0,1/(n_M+2))$, there exists a sufficiently large positive integer $k$ such that
\begin{align*}
W_p\left(\delta_{(x_0,i_0)}\mathbf{P}_{k\Delta}, \pi\right) + W_p\left(\delta_{(x_0,i_0)}\mathbf{P}_{k}^\Delta, \pi^\Delta\right) \leq C e^{-\gamma^* k\Delta}, 
\end{align*}
where $\gamma^* := \gamma \wedge \gamma_{\Delta}$. Following the similar way of the proof of Theorem 4.2 in \cite{LiuMaoWu2023}, there exists a $\Delta^*>0$ such that for $\Delta \in (0,\Delta^*)$ 
\begin{align*}
W_p\left(\delta_{(x_0,i_0)}\mathbf{P}_{k\Delta}, \delta_{(x_0,i_0)}\mathbf{P}_{k}^\Delta\right) \leq C  \Delta^{\frac{p}{2}}.
\end{align*}
Finally, we can choose $N= [- (p ln \Delta)/(2\gamma^*\Delta)]$. Then, for any $k > N$,  we have 
\begin{align*}
W_p\left(\pi, \pi^\Delta\right) \leq C \Delta^{\frac{p}{2}}+C e^{-\gamma^* k\Delta} \leq C \Delta^{\frac{p}{2}}.
\end{align*}
The proof is complete.

\end{proof}

%

\section{Numerical simulations}\label{sec:ns}
A scalar SDE with Markovian switching is employed as an example to demonstrate the theoretical results empirically in this section.
\begin{example}\label{ex:1}
For a Markov process $\{r(t)\}_{t \geq 0}$ taking values in $\mathbb{S} = \{1, 2\}$ with the transition rate matrix 
$$\Gamma = \begin{pmatrix} -4 & 4 \\ 1 & -1 \end{pmatrix},$$
we consider a hybrid SDE \eqref{theSDE} with 
$$f(x,1)=1+x-x^3, g(x,1)=x^2~\text{and}~f(x,2)=1-2x-3x^3, g(x,2)=-x^2.$$
Set initial data $X(0) = 2$ and $r(0) = 1$. 
\end{example}

Firstly, for the equation described in Example \ref{ex:1} we check if all the required assumptions are obeyed.
\par
Due to $x^3-y^3=(x-y)(x^2+xy+y^2)$ and $x^2+xy+y^2\le \frac{3}{2}(|x|^2+|y|^2)$, for any $i \in \mathbb{S}$ we have  
\begin{equation*}
|f(x,i)-f(y,i)|
\le C\,|x-y|\left(1+|x|^2+|y|^2\right),
\end{equation*}
where $C>0$ is a constant. In addition, for any $i \in \mathbb{S}$ we have
\begin{equation*}
|g(x,i)-g(y,i)|
=|x^2-y^2|
=|x-y||x+y|
\le C\,|x-y|\sqrt{|x|^2+|y|^2}.
\end{equation*}
Hence Assumption \ref{PolyCon} holds with $q=6$.
\par
By choosing $l_1=5$, we have
\begin{align*}
&2(x-y)\bigl(f(x,1)-f(y,1)\bigr)+l_1|g(x,1)-g(y,1)|^2
\le 2|x-y|^2,    
\end{align*}
and
\begin{align*}
&2(x-y)\bigl(f(x,2)-f(y,2)\bigr)+l_1|g(x,2)-g(y,2)|^2
\le -4|x-y|^2.    
\end{align*}
That is, Assumption \ref{KhCon} holds with $n_1=2$ and $n_2=-4$. It can be derived from the $\Gamma$ that the stationary distribution of the Markov chain is $\mu=(0.2,0.8)$. Thus, Assumption \ref{invCon} is satisfied.
\par
As all the requirements for those theorems in Section \ref{sec:MRs} is satisfied, we can conclude that for the numerical solution generated by the BEM method has a unique invariant measure, which, moreover, converges to the underlying counterpart.
\par
Now we conduct some numerical simulations to display them. We set the time step-size $\Delta =0.01$.
\par
The left plot of Figure \ref{fig:r} depicts a sample path of the Markov
chain $r(t)$ and the right plot shows the visits to each state of that sample path.
\begin{figure}[ht]  
    \centering  
    \begin{minipage}[b]{0.49\textwidth}
        \centering        \includegraphics[width=\textwidth]{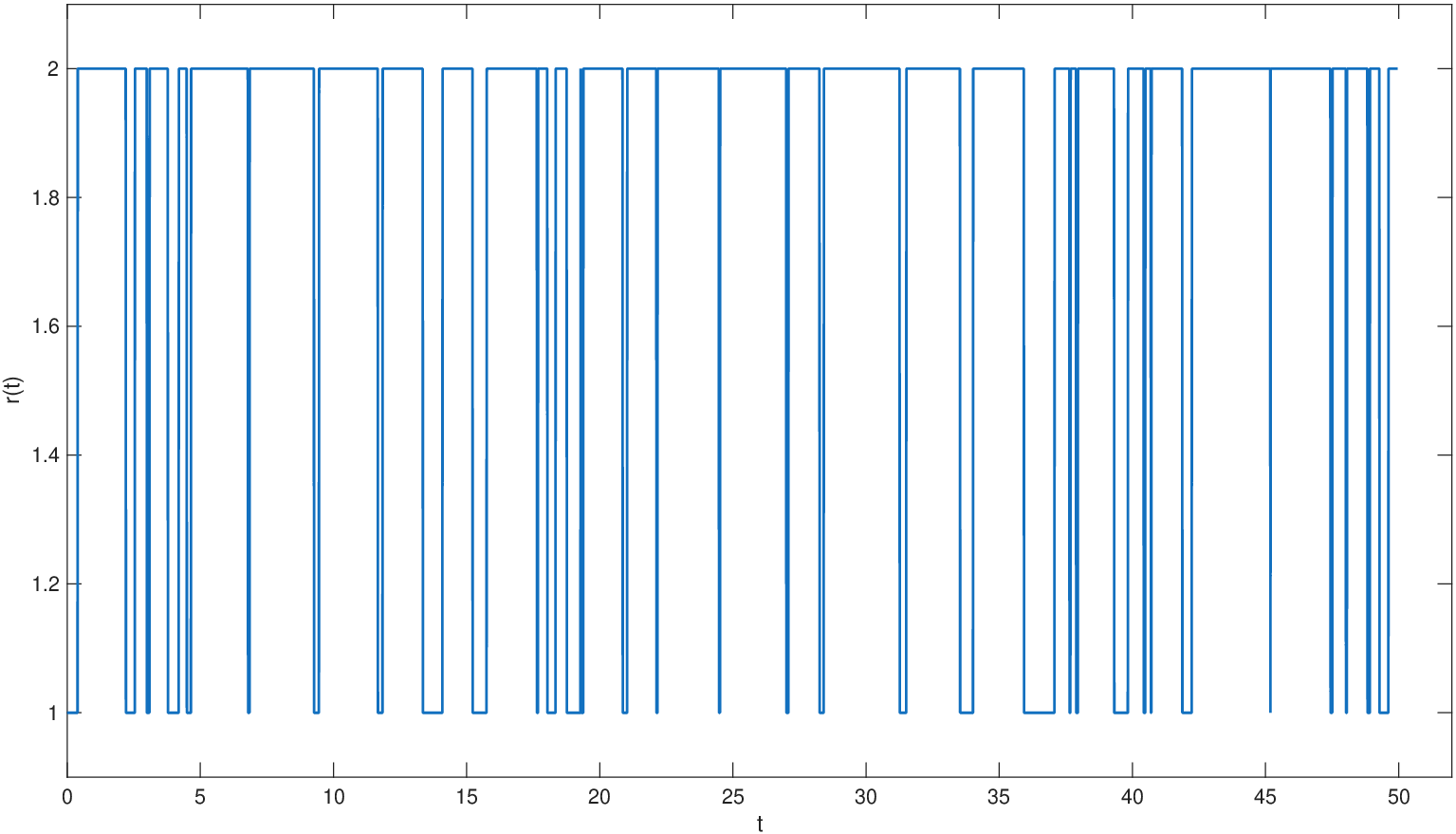}
    \end{minipage}
    \hfill
    \begin{minipage}[b]{0.49\textwidth}
        \centering        \includegraphics[width=\textwidth]{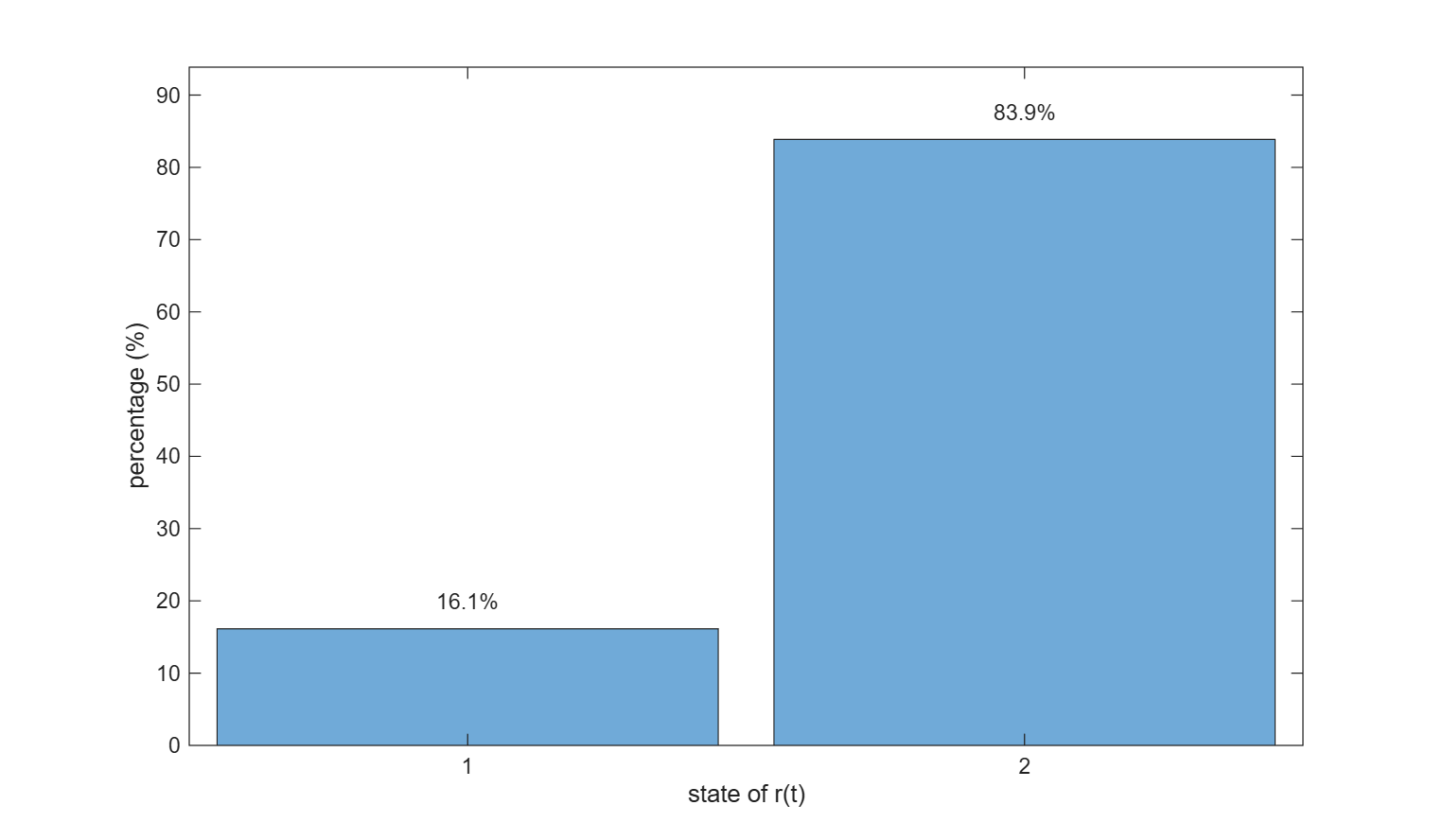}
    \end{minipage}
    \caption{Left: one sample path of $r(t)$. Right: proportions of visits to each state for one path of $r(t)$  }  
      \label{fig:r}
\end{figure}

1000 sample paths of the equation in Example \ref{ex:1} are simulated. The empirical densities at different time points are displayed in the left plot of Figure \ref{fig:main}. At $t = 0.1$, $t = 0.3$, and $t = 0.5$, the shapes of the empirical density functions are notably distinct. However, when we look at $t = 20$ and $t = 30$, the empirical density functions become much more similar. This similarity implies the existence of an invariant measure for the system. 
\par
In the right plot of Figure  \ref{fig:main}, the empirical densities at $t=40$ for different initial value $X(0)=-5,5,15$ are drawn. The virtual identity of those three curves illustrates the uniqueness of the invariant measure.

\begin{figure}[htbp]
    \centering
    \begin{minipage}[b]{0.49\textwidth}
        \centering        \includegraphics[width=\textwidth]{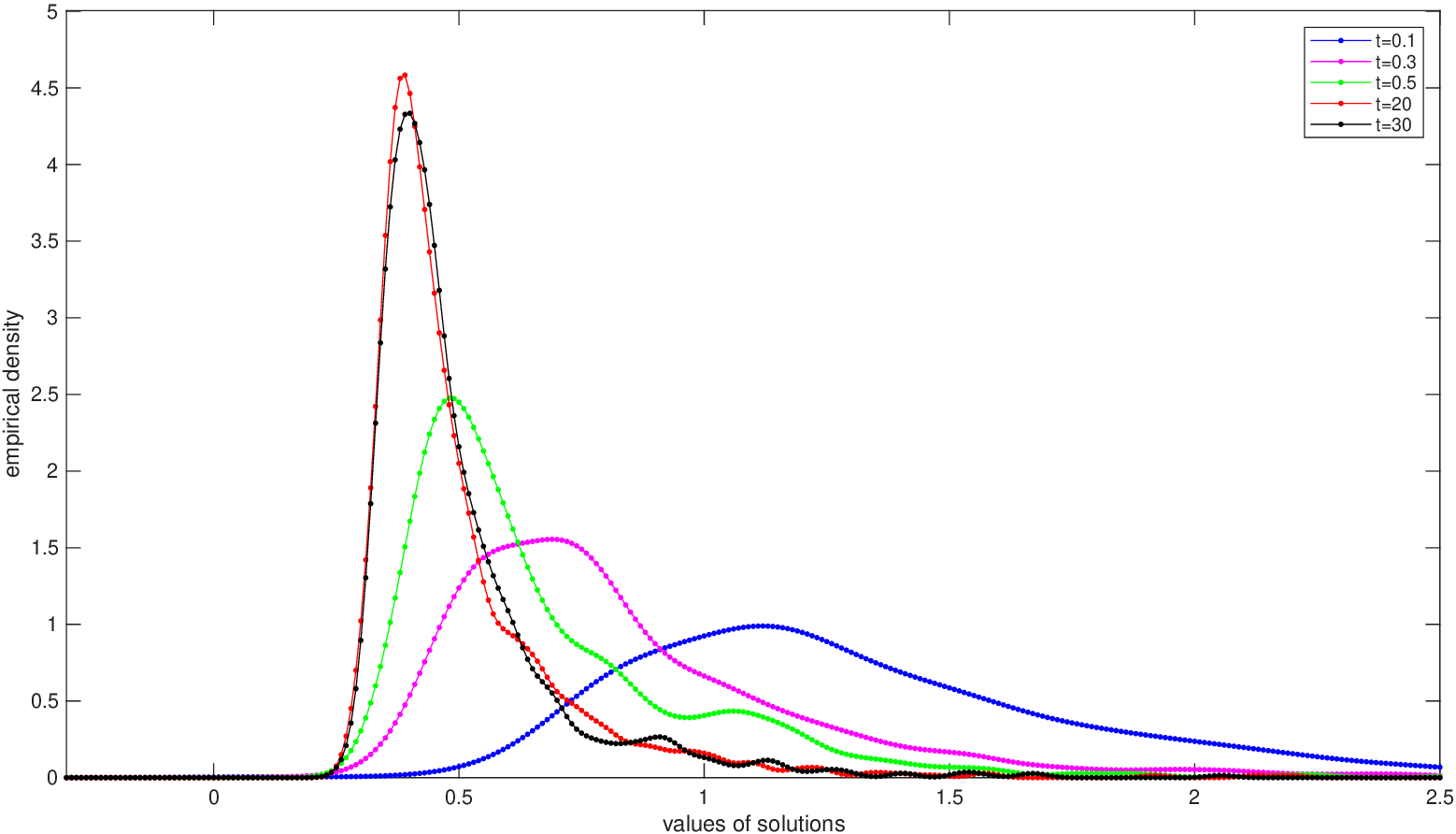}
    \end{minipage}
    \hfill
    \begin{minipage}[b]{0.49\textwidth}
        \centering
        \includegraphics[width=\textwidth]{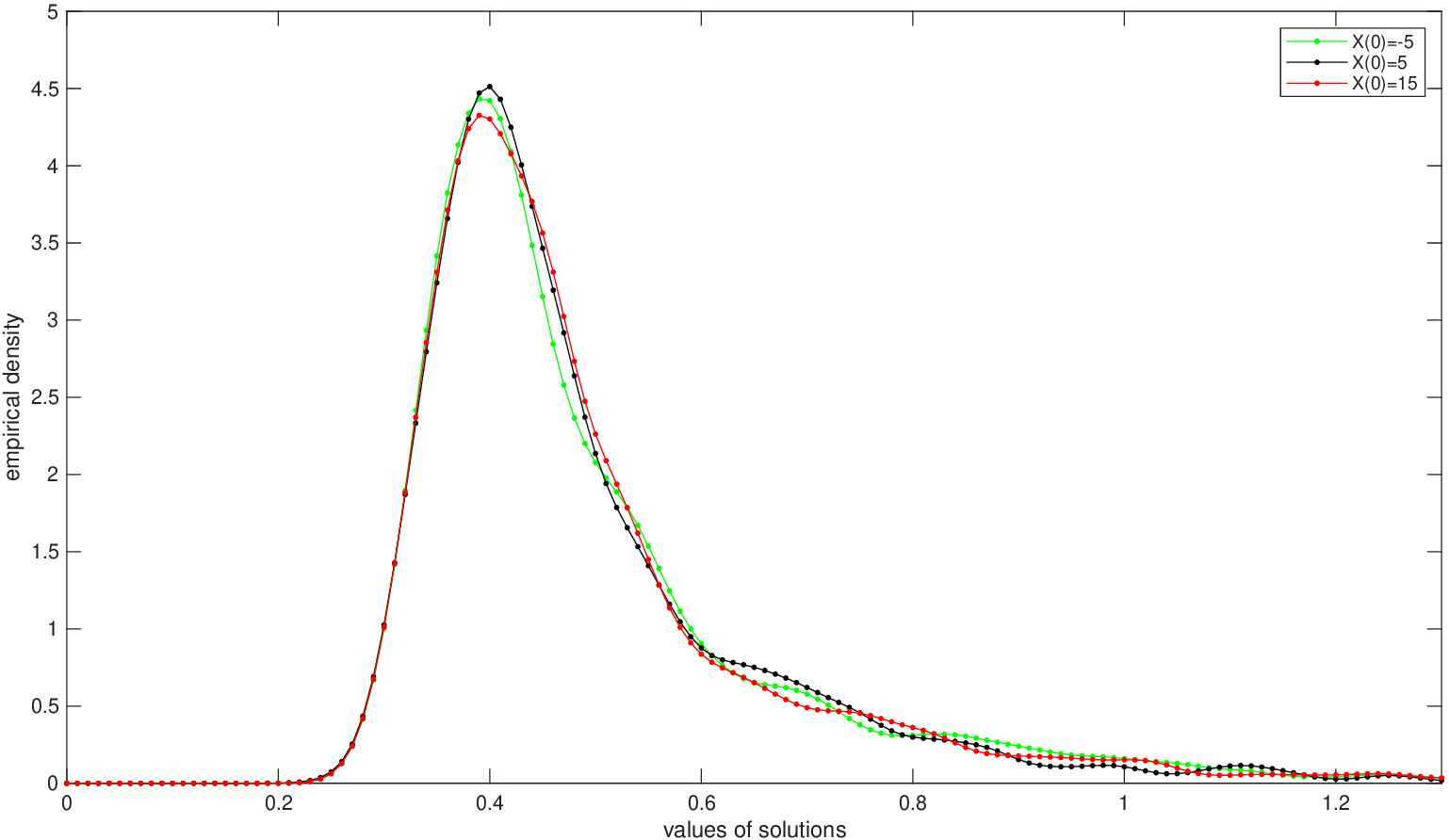}
    \end{minipage}
    \caption{Left: empirical density functions at different time points. Right: empirical density functions at t=40 for SDEs with different initial values.}
    \label{fig:main}
\end{figure}

To quantitatively measure the similarity of those empirical density functions, we apply two-sample Kolmogorov-Smirnov (K-S) tests between samples of solutions at consecutive time points $t_i=i\Delta$ and $t_{i+1}=(i+1)\Delta$ for $i=0,\cdots,200$.
\par
From the top plot in Figure \ref{fig:2}, the K-S statistics drop to near zero quite fast and remain small. There indicate that the empirical density functions converge to the invariant one rapidly. 
\par
The bottom plot in Figure \ref{fig:2} shows that $p$-values rise quickly and stay much larger 0.05 for rest of the time. These provide the strong confidence that samples of solutions at consecutive time points come from the same distribution.

\begin{figure}[ht]  
    \centering    \includegraphics[width=0.75\textwidth]{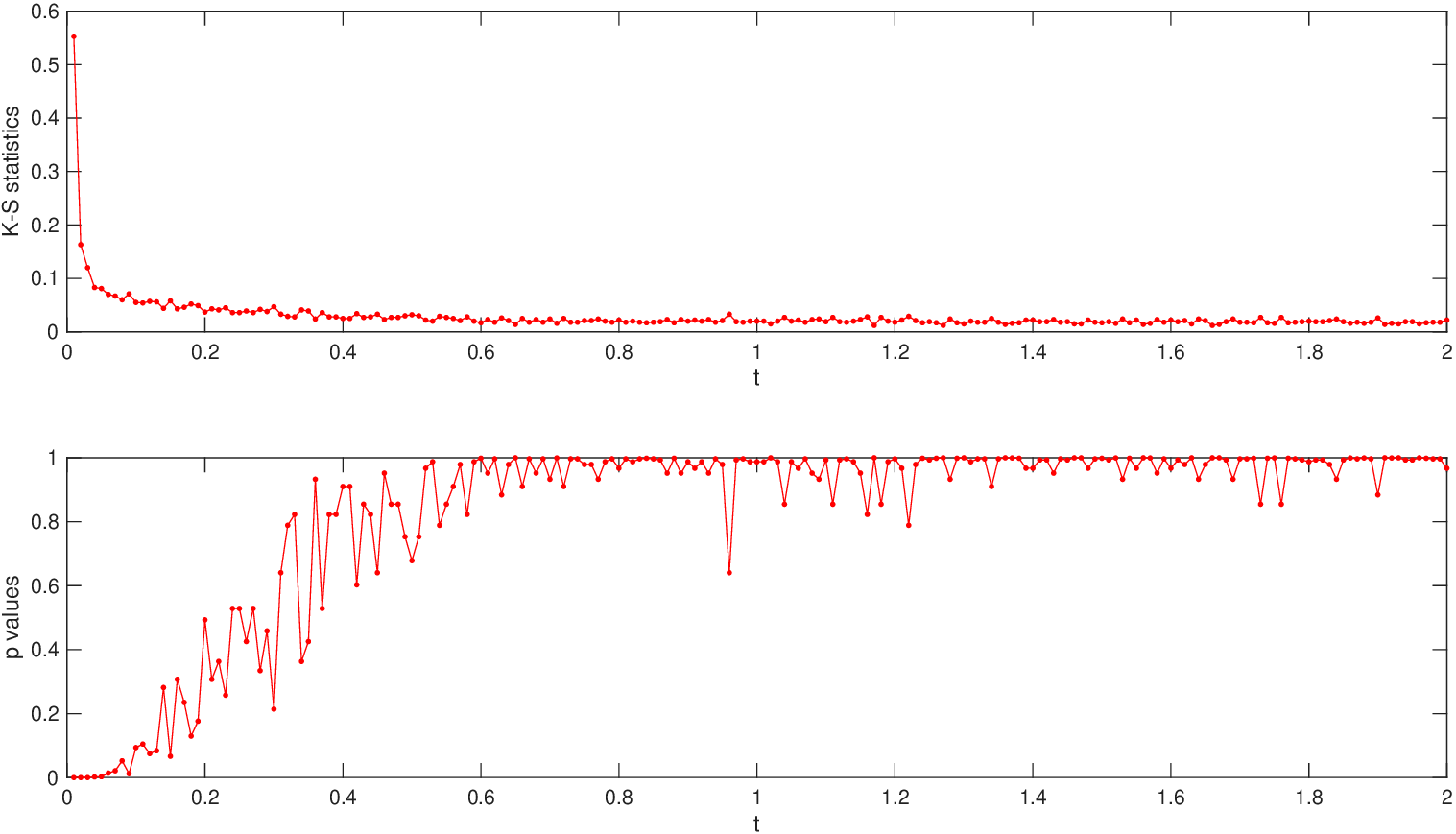}  
    \caption{K-S tests for samples at consecutive time points.}  
    \label{fig:2}  
\end{figure}

\section{Conclusion and future works}
In this paper, the BEM method is employed to approximate the invariant measure of SDE with Markovian switching. Both the drift and the diffusion coefficients of the underlying equation are allowed to contain super-linear terms. Compared with existing works, our results release the constraint of the global Lipschitz condition on the diffusion coefficient.  
\par
As the invariant measure can also be derived from some partial differential equation that itself also needs some numerical treatments, our results in this work, on the other way around, provide some possibility to numerically solve those partial differential equations via some stochastic algorithms. Therefore, this is one of the future works we are thinking about.

\bibliography{RefLists}

\end{document}